\documentclass{article}
\RequirePackage{amsthm,amsmath}
\usepackage{amsfonts}
\usepackage{amssymb}
\usepackage{fancyvrb}
\DeclareSymbolFont{AMSb}{U}{msb}{m}{n}
\DeclareMathSymbol{\N}{\mathbin}{AMSb}{"4E}
\DeclareMathSymbol{\Z}{\mathbin}{AMSb}{"5A}
\DeclareMathSymbol{\R}{\mathbin}{AMSb}{"52}
\DeclareMathSymbol{\Q}{\mathbin}{AMSb}{"51}
\DeclareMathSymbol{\I}{\mathbin}{AMSb}{"49}
\DeclareMathSymbol{\C}{\mathbin}{AMSb}{"43}

\newtheorem{theorem}{Theorem}
\newtheorem{lemma}{Lemma}

\title{Expectation of Stratonovich iterated integrals of Wiener processes}
\author{Christophe Ladroue\\Department of Statistics\\University of Warwick, UK}
\begin{document}
\maketitle{}
\begin{abstract}
The solution of a (stochastic) differential equation (SDE) can be locally approximated by a stochastic expansion, a linear combination of iterated integrals. Quantities of interest, like moments, can then be approximated with the expansion. We present a formula for the case where the drivers of the equation are time and Wiener processes. We also present a \texttt{Mathematica} implementation of the result.
\end{abstract}

\section{Introduction}
We consider the stochastic differential equation
$$
dY_t=f(Y_t)dt+\sum_{i=1}^{n-1}g(Y_t)dW_t^i
$$

where the drivers $W^i$ are Wiener processes. By convention, we set time as the zeroth driver: $dt=dW_t^0$. We call $J_\alpha(t)$ the Stratonovich iterated integral according to the multi-index $\alpha$, $\alpha_i$ running from $0$ to $n-1$. We call $I_\alpha(t)$ the It\^o iterated integral according to the multi-index $\alpha$:
\begin{eqnarray*}
I_\alpha(t)&=& \int \hdots \int_{0< u_1< \hdots < u_k<t}dW_{u_1}^{\alpha_1}\hdots dW_{u_\ell}^{\alpha_\ell}\\
J_\alpha(t)&=& \int \hdots \int_{0< u_1< \hdots < u_k<t}\circ dW_{u_1}^{\alpha_1}\hdots \circ dW_{u_\ell}^{\alpha_\ell}
\end{eqnarray*}
with $|\alpha|=\ell$.

Our aim is to arrive at a formula for the expectation of $J_\alpha$ for all $\alpha$. We will use two well-known results: the expectation of It\^o iterated integrals $I_\alpha$ is very easy to calculate and a recursive formula link the two iterated integrals. Specifically:
\begin{enumerate}
\item For Wiener processes, $E I_\alpha=0$ if $\exists i$ such that $\alpha_i\neq0$. Otherwise, $E I_\alpha(t)=t^\ell/\ell!$
\item $J_\alpha=\int J_{\alpha-}dW^{\alpha_\ell}+\frac{1}{2}\chi(\alpha_{\ell-1}=\alpha_\ell\neq0)\int J_{\alpha--}ds$ (Lemma 3, \cite{Tocino2009})
\end{enumerate}

From this, we can derive a closed formula for the expectation of any $J_\alpha$.

\section{Formula for $E J_\alpha$}
\begin{lemma}\label{lemma:decomposition}
$J_\alpha=\sum c_\beta I_\beta+p_\alpha I_{0,0,\dots, 0}$
where the first sum is done over some words $\beta$ which all contain at least $1$ non-zero letter. $c_\beta$ and $p_\alpha$ are reals, $p_\alpha$ can be zero. We denote $q_\alpha$  the number of zeros in the last It\^o integral. 
\end{lemma}

All quantities depend on $\alpha$ but the exact values of $\beta$ and $c_\beta$ will be of no interest for us here. In other words, the lemma states that $J_\alpha$ can be written as linear combination of It\^o iterated integrals, with \emph{at most} one of them of the form $I_{0,\dots,0}$.
The proof comes easily by recursion on the length of $\alpha$, allowing $p_\alpha$ to be equal to $0$. This result holds for all type of drivers; it is not limited to Wiener processes.
\begin{proof}
The lemma is true for $\ell\leq2$:
$$\begin{array}{lclclc}
J_0&=&I_0&&\\
J_1&=&I_1&&\\
J_{0,0}&=&\int I_0ds+0&=&I_{0,0}\\
J_{1,0}&=&\int I_1ds+0&=& I_{1,0}\\
J_{0,1}&=&\int I_0dW^1+0&=&I_{0,1}\\
J_{1,1}&=&\int I_1dW^1+\frac{1}{2}\int 1 ds&=&I_{1,1}+\frac{1}{2}I_0
\end{array}
$$

Assume the lemma is true for words up to length $<\ell$ and consider a word $\alpha$ of length $\ell$. Three cases are possible:
\begin{itemize}
\item $\alpha_\ell=0$
	We have:
\begin{eqnarray*}
J_\alpha&=&\int J_{\alpha-} ds+0\\
&=& \int (\sum c_\beta I_\beta+p_{\alpha-} I_{0,\hdots,0})ds\\
&=& \sum c_\beta I_{\beta,0}+p_{\alpha-} I_{0,\hdots,0,0}
\end{eqnarray*}
where all $\beta$'s have at least 1 non-zero letter.

\item $\alpha_{\ell-1}\neq\alpha_\ell\neq0$
We have:
\begin{eqnarray*}
J_\alpha&=&\int J_{\alpha-} dW^{\alpha_\ell}+0\\
&=& \int (\sum c_\beta I_\beta+p_{\alpha-} I_{0,\hdots,0})dW^{\alpha_\ell}\\
&=& \sum c_\beta I_{\beta,\alpha_\ell}+p_{\alpha-} I_{0,\hdots,0,\alpha_\ell}\\
&=& \sum c_{\beta'} I_{\beta'}
\end{eqnarray*}

\item $\alpha_{\ell-1}=\alpha_\ell\neq0$
We have:
\begin{eqnarray*}
J_\alpha&=&\int J_{\alpha-} dW^{\alpha_\ell}+\frac{1}{2}\int J_{\alpha--}ds\\
&=& \int (\sum c_\beta I_\beta+p_{\alpha-} I_{0,\hdots,0})dW^{\alpha_\ell}+\frac{1}{2}\int (\sum c_{\beta'} I_{\beta'}+p_{\alpha--} I_{0,\hdots,0})ds\\
&=& \sum c_{\beta''} I_{\beta'',\alpha_\ell}+\frac{1}{2}p_{\alpha--} I_{0,\hdots,0,0}
\end{eqnarray*}

\end{itemize}

\end{proof}

\begin{lemma} 
For Wiener processes, $EJ_\alpha(t)=p_\alpha .t^{q_\alpha}/q_\alpha!$ or $0$.
\end{lemma}

This follows directly from lemma \ref{lemma:decomposition} by taking the expectation on both sides. $E J_\alpha\neq0$ iff $p_\alpha\neq0$, \emph{i.e.} iff $J_\alpha$ has an $I_{0,\dots,0}$ component.

\begin{lemma}
For Wiener processes, $EJ_\alpha\neq0$ iff $\alpha$ is a sequence of either $0$ or pairs $mm$, where $m$ is a driver $\neq0$.
\end{lemma}
\begin{proof}
For $J_\alpha$ to have an $I_{0,\dots,0}$ component, it is necessary to have:
\begin{itemize}
\item[a)] $\alpha_\ell=0$ and $J_{\alpha-}$ to have an $I_{0,\dots,0}$ component.
\item[b)] $\alpha_\ell\neq0$ and $\alpha_{\ell-1}=\alpha_\ell$ and $J_{\alpha--}$ to have an $I_{0,\dots,0}$ component.
\end{itemize}
Note that a) and b) are mutually exclusive.
\end{proof}

\begin{lemma}
When $p_\alpha\neq0$, $p_\alpha=1/2^\frac{\#\{\alpha_i\neq0\}}{2}$ and $q_\alpha=(\#\{\alpha_i\neq0\})/2+(\#\{\alpha_i=0\})$
\end{lemma}

\begin{proof}
This comes from the proof for result 3: if we are in case a), we integrate $J_{\alpha-}$ with respect to time, so $q_\alpha=q_{\alpha-}+1$. In case b), we integrate $J_{\alpha--}$ with respect to time and divide by 2, so $q_\alpha=q_{\alpha--}+1$ and $p_\alpha=p_\alpha/2$. 
\end{proof}

\begin{theorem}\label{theorem:mainTheorem}
Given a word $\alpha$ and assuming Wiener processes:
$$EJ_\alpha(t)=\left\{
\begin{array}{l}
0\textnormal{ if }\alpha \textnormal{ is not a sequence of } 0 \textnormal{ and pairs } mm\\
p_\alpha\frac{t^{q_\alpha}}{q_\alpha!}\textnormal{ as defined in result 4 otherwise.}
\end{array}\right.$$
\end{theorem}

Thus, for example:
\\$\begin{array}{lclcl}
E J_{0,1,1,0,0}&=&1/2^{2/2} t^{(3+2/2)}/(3+2/2)!&=&\frac{1}{2}\frac{t^4}{4!}\\
E J_{0,1,1,0,0,1}&=&0&&\\
E J_{2,2,1,1,3,3}&=& 1/2^{6/2} t^{(0+6/2)}/(0+6/2)!&=&\frac{1}{2^3}\frac{t^3}{3!}\\
E J_{2,2,0,1,1,3,3,0,0,0}&=& 1/2^{6/2} t^{(4+6/2)}/(4+6/2)!&=&\frac{1}{2^3}\frac{t^7}{7!}\\
\end{array}
$

\section{\texttt{Mathematica} implementation}
The result in theorem \ref{theorem:mainTheorem} can easily be implemented in \texttt{Mathematica}. This iterative implementation is to replace a recursive one, as it requires less memory and is faster, computing the expectation in at most $\ell$ iterations (where $\ell$ is the size of word $\alpha$).
 
\begin{Verbatim}[numbers=left]
ExpSBM[t_, j[a_List]] := Module[{i, c},
   i = Length@a;
   c = {0, 0};
   Catch[
    While[i > 0,
     If[a[[i]] == 0,
      c += {0, 1}; i--,
      If[(i > 1) && (a[[i]] == a[[i - 1]]),
       c += {1, 1}; i -= 2,
       c = {Infinity, 0}; Throw@0
       ]]]];
   (1/2)^First@c t^Last@c/(Last@c)!];
\end{Verbatim}
Using this code yields the expected results:
\begin{Verbatim}
{ExpSBM[t, j[{0, 1, 1, 0, 0}]],
 ExpSBM[t, j[{0, 1, 1, 0, 0, 1}]],
 ExpSBM[t, j[{2, 2, 1, 1, 3, 3}]],
 ExpSBM[t, j[{2, 2, 0, 1, 1, 3, 3, 0, 0, 0}]]}

{t^4/48, 0, t^3/48, t^7/40320}
\end{Verbatim}

\section{Conclusion}
Using a relation between Stratonovich and It\^o integrals, we derived a simple formula for the expectation of Stratonovich iterated integrals when the drivers are Wiener processes. This result was then implemented in \texttt{Mathematica}.

\bibliography{strato_expectation}

\begin{thebibliography}{9}

\bibitem{Tocino2009}
  A. Tocino.
{Multiple stochastic integrals with mathematica.},
	\emph{Mathematics and Computers in Simulation}, 79(5):1658–1667, January 2009.
\end{thebibliography}

\end{document}